\documentclass[12pt]{article}
\usepackage[T1]{fontenc}
\usepackage{amsmath}
\usepackage{amsthm}
\usepackage{stmaryrd}
\usepackage{amsfonts}
\usepackage{amssymb}
\usepackage{graphicx}
\usepackage{subcaption}
\usepackage{fullpage}
\usepackage{setspace}
\usepackage{dsfont}
\usepackage{chngcntr}
\usepackage{xcolor}
\usepackage{hyperref}
\hypersetup{
	colorlinks=true,
	linkcolor=blue,
	citecolor=red,
	urlcolor=blue,
	pdfauthor={Etienne MARION},
	pdfpagemode=FullScreen,
}

\newcommand{\lean}[1]{\lstinline[language=lean]{#1}}

\definecolor{keywordcolor}{rgb}{0.7, 0.1, 0.1}   
\definecolor{tacticcolor}{rgb}{0.0, 0.1, 0.3}    
\definecolor{commentcolor}{rgb}{0.4, 0.4, 0.4}   
\definecolor{stringcolor}{rgb}{0.5, 0.3, 0.2}    
\definecolor{symbolcolor}{rgb}{0.1, 0.2, 0.7}    
\definecolor{sortcolor}{rgb}{0.1, 0.5, 0.1}      
\definecolor{attributecolor}{rgb}{0.7, 0.1, 0.1} 
\definecolor{errorcolor}{rgb}{1, 0, 0}           

\usepackage{listings}

\usepackage{scalefnt}

\newcommand{\A}{\mathcal{A}}
\newcommand{\B}{\mathcal{B}}
\newcommand{\C}{\mathcal{C}}
\newcommand{\E}{\mathbb{E}}
\newcommand{\F}{\mathcal{F}}
\newcommand{\N}{\mathbb{N}}
\newcommand{\R}{\mathbb{R}}

\newcommand{\set}[1]{\left\{#1\right\}}
\newcommand{\inv}{^{-1}}
\newcommand{\eps}{\epsilon}

\newcommand{\Xge}[1]{X^{\ge#1}}
\newcommand{\Xle}[1]{X^{\le#1}}
\newcommand{\Xgt}[1]{X^{>#1}}
\newcommand{\Age}[1]{\mathcal{A}^{\ge#1}}

\newcommand{\Ale}[1]{\mathcal{A}^{\le#1}}
\newcommand{\Xgtn}{X^{>n}}
\newcommand{\Xgtb}{X^{>b}}

\newcommand{\Xlen}{X^{\le n}}

\newcommand{\Xlea}{X^{\le a}}
\newcommand{\Alen}{\mathcal{A}^{\le n}}
\newcommand{\dx}{\mathrm{d}x}
\newcommand{\dy}{\mathrm{d}y}
\newcommand{\dz}{\mathrm{d}z}
\newcommand{\prth}[1]{\left(#1\right)}
\renewcommand{\empty}{\varnothing}
\newcommand{\tendl}[1]{\longrightarrow_{#1}}
\newcommand{\dbrack}[1]{\llbracket #1 \rrbracket}

\newcommand{\mub}{\bar{\mu}}
\newcommand{\ox}{\otimes}
\renewcommand{\phi}{\varphi}
\newcommand{\piun}{{\pi_{\dbrack{1,n}}}}

\newcommand{\Xint}[1]{X^{\dbrack{#1}}}
\newcommand{\Aint}[1]{\A^{\dbrack{#1}}}
\newcommand{\kernel}{\rightsquigarrow}
\newcommand{\ind}[1]{\mathds{1}_{#1}}

\newcommand{\mathlib}{\texttt{Mathlib}}
\newcommand{\oxm}{\ox_m}
\newcommand{\oxk}{\ox_k}
\newcommand{\x}{\times}
\renewcommand{\o}{\circ}

\newtheorem{thm}{Theorem}[section]
\newtheorem{prop}[thm]{Proposition}

\newtheorem{lem}[thm]{Lemma}
\theoremstyle{definition}
\newtheorem{defi}[thm]{Definition}

\theoremstyle{remark}
\newtheorem*{nota}{Notation}

\begin{document}

	\begin{center}
		{\Large\bf A Formalization of the Ionescu-Tulcea Theorem in Mathlib} \\
		\vspace{1cm}
		Etienne Marion \\
		\vspace{1cm}
		MAP5, Université Paris-Cité, 45 rue des Saints-Pères, 75006 Paris, France \\
		\texttt{etienne.marion@math.cnrs.fr}
	\end{center}

	\vspace{0.5cm}

	\section*{Abstract}
	We describe the formalization of the Ionescu-Tulcea theorem, showing the existence of a probability measure on the space of trajectories of a Markov chain, in the proof assistant Lean using the integrated library \mathlib. We first present a mathematical proof before exposing the difficulties which arise when trying to formalize it, and how they were overcome. We then build on this work to formalize the construction of the product of an arbitrary family of probability measures.

	\section{Introduction}
	Being able to talk about the joint distribution of an infinite family of random variables is crucial in probability theory. For example, one often needs the existence of a family of independent random variables. This relies on the existence of an infinite product measure. Indeed, given $(\Omega_i, \F_i, \mu_i)_{i\in\iota}$ a family of probability spaces, the existence of the product measure $\bigotimes_{i\in\iota}\mu_i$ yields a new probability space $(\prod_{i\in\iota}\Omega_i, \bigotimes_{i\in\iota}\F_i, \bigotimes_{i\in\iota}\mu_i)$, and the projections $X_i : \prod_{j\in\iota}\Omega_j \to \Omega_i$ give the desired family. For another example, consider discrete-time Markov chains. Given a measurable space $(E, \A)$ and a Markov kernel $\kappa$ from $E$ to $E$, one might want to build a sequence $(X_n)_{n\in\N}$ of random variables with values in $E$ such that the conditional distribution of $X_{n+1}$ given $X_0, ..., X_n$ is $\kappa(X_n,\cdot)$.

	Such objects are fundamental in probability theory: families of independent variables allow to build more complicated objects, such as Brownian motion, while discrete-time Markov chains form a huge class of stochastic processes which contains random walks for instance. It so happens that those objects always exist without any restrictions on the spaces we consider. This is a direct consequence of the Ionescu-Tulcea theorem.

	The goal of this contribution is to provide a formalization of the proof of the aforementioned theorem using the proof assistant Lean \cite{Lean} and the associated library \mathlib~\cite{Mathlib}. The code is now fully integrated to \mathlib\ and is therefore available to anybody who wishes to do formalization using Lean (see \href{https://github.com/leanprover-community/mathlib4/blob/c1f268b3370e727834bc31ca014d3c212eb91b42/Mathlib/Probability/Kernel/IonescuTulcea/Traj.lean#L510-L520}{here} for the definition of the main object and \href{https://github.com/leanprover-community/mathlib4/blob/c1f268b3370e727834bc31ca014d3c212eb91b42/Mathlib/Probability/ProductMeasure.lean#L356}{here} for the definition of the infinite product of probability measures). It relies on previous work done by Rémy Degenne and Peter Pfaffelhuber who formalized the Carathéodory's extension theorem as well as projective families of measures and measurable cylinders.

	The Ionescu-Tulcea theorem has already been formalized by Johannes Hölzl in the proof assistant Isabelle/HOL \cite{Holzl2016} and the code is available \href{https://isabelle.in.tum.de/dist/library/HOL/HOL-Probability/Projective_Family.html#Projective_Family.Ionescu_Tulcea|locale}{here}. The main difference is that dependent types do not exist in Isabelle/HOL. Therefore Johannes Hölzl only consider measures over $X^\N$ for some measurable space $X$ rather than over a general product space $\prod_{n \in \N} X_n$. One of the main issues we stumbled upon was the so-called subtype-hell, and several of our design decisions were made to avoid manipulating subtypes too much, as explained in Section~\ref{sec:forma}. This issue does not occur in Isabelle/HOL because there are no subtypes. Indeed, given a predicate \lean{p : X → Prop}, the associated subtype consists of dependent pairs \lean{<x, h>}, where \lean{h} is a proof of \lean{p x}. Thus a subtype is a special case of a dependent type, so does not exist in non-dependent type theory. Because of this Hölzl's construction makes certain manipulations easier, while our construction seems to be more practical, as further	discussed in Section~\ref{sub:comp-ker}. Note also that in the end, both the formalized version by Hölzl and the Theorem~8.24 in the classical reference \cite{Kal2021} only build a measure over the product space, while we build a family of Markov kernels from $\prod_{i \le a} X_i$ to $\prod_{n\in\N} X_n$. This in particular provides a more practical framework to build Markov chains.

	We build on the formalization of the Ionescu-Tulcea theorem to formalize the existence of arbitrary products of probability measures. This construction was already accomplished in Rocq \cite{Shin2022}, although not via the Ionescu-Tulcea theorem.

	In Section~\ref{sec:maths} we present basic notions and results of probability theory as well as the mathematical statement and proof of Ionescu-Tulcea theorem. Then in Section~\ref{sec:forma} we present the key points of the theorem which proved complicated to formalize. Finally in Section~\ref{sec:prod-measure} we briefly describe how this work allows to formalize the product of an arbitrary family of probability measures.

	\section{Mathematical aspects}\label{sec:maths}

	\subsection{Preliminaries on probability theory}\label{sec:proba}

	The goal of this paper is to illustrate how formalization sometimes forces us to move away from the classical approach. Therefore in this section we recall this approach by providing standard definitions and results from measure theory and probability theory. We do not necessarily give fully general results and only state results which will be used later. For a complete reference, one can for instance take a look at \cite{Kal2021}.

	We first recall Carathéodory’s extension theorem, which plays a key role in the proof of the Ionescu-Tulcea theorem.

	\begin{defi}\label{def:ring}
		Let $X$ be a set. A \emph{ring of sets} over $X$ is a family of subsets of $X$ which contains the empty set and is closed under finite intersections and set difference.
	\end{defi}

	We now fix a set $X$ and a ring of sets $S$ over $X$.

	\begin{defi}\label{def:add-cont}
		An \emph{additive content} over $S$ is a map $\mu : \mathcal{P}(X) \to [0,\infty]$ such that $\mu(\empty) = 0$ and if $A_1, ..., A_n$ are pairwise disjoint sets in $S$ such that $\bigcup_{i=1}^n A_i \in S$, then we have
		$\mu\prth{\bigcup_{i=1}^n A_i} = \sum_{i=1}^n\mu(A_i)$.
	\end{defi}

	\begin{defi}\label{def:sigma-sub}
		Let $\mu$ be an additive content over $S$. Then $\mu$ is called \emph{$\sigma$-sub-additive} if for any $(A_n)_{n\in\N}$ such that for all $n\in\N$, $A_n \in S$ and $\bigcup_{n\in\N} A_n \in S$, we have the inequality $\mu\prth{\bigcup_{n\in\N} A_n} \le \sum_{n\in\N} \mu(A_n)$.
	\end{defi}

	\begin{lem}\label{lem:cont-empty}
	Let $\mu$ an additive content over $S$. To prove that $\mu$ is $\sigma$-sub-additive, it is enough to show that it is continuous at the empty set, namely that if $(A_n)_{n\in\N}$ is a non-increasing (for the inclusion of sets) family of elements of $S$ with empty intersection, then $\mu(A_n) \tendl{n\to\infty} 0$.
	\end{lem}

	\begin{thm}[Carathéodory's extension theorem]\label{thm:carath}
		Let $\mu$ be an additive content over $S$ which is also $\sigma$-sub-additive. Then there exists a measure, defined over the smallest $\sigma$-algebra containing $S$, which extends $\mu$.
	\end{thm}

We now give the definitions of basic notions related to Markov kernels, which are the objects of concern in the Ionescu-Tulcea theorem. For a more detailed account of transition kernels and how they are used in \mathlib, see \cite{Degenne25}.

	\begin{defi}\label{def:markov-kernel}
	Let $(X,\A)$ and $(Y,\B)$ be two measurable spaces. A \emph{Markov kernel} from $X$ to $Y$ is a map $\kappa : X \times \B \to [0,1]$ such that:
	\begin{itemize}
		\item for any $x \in X$, $\kappa(x,\cdot)$ is a probability measure;
		\item for any $B \in \B$, $\kappa(\cdot,B)$ is measurable.
	\end{itemize}
	One can therefore consider a Markov kernel as a measurable map which to any point in $X$ associates a probability distribution over the space $Y$, or more intuitively a random point in $Y$. In what follows we will often refer to $\kappa(x, \cdot)$ by simply writing $\kappa(x)$. Also, to denote the fact that $\kappa$ is a Markov kernel from $X$ to $Y$ we write $\kappa : X \kernel Y$.
\end{defi}

	There exist many operations involving Markov kernels, we now give the ones that will be used later.

	\begin{defi}\label{def:ker-map}
		Let $(X, \A)$, $(Y, \B)$ and $(Z, \C)$ be three measurable spaces, $\kappa : X \kernel Y$ and $f : Y \to Z$ be a measurable function. The \emph{push-forward kernel} of $\kappa$ by $f$ is the Markov kernel $f_*\kappa : X \kernel Z$ defined by:
		$$\forall x \in X, \forall C \in \C, f_*\kappa(x, C) := \kappa(x, f\inv(C)).$$
	\end{defi}

	\begin{defi}\label{def:comp}
		Let $(X, \A)$, $(Y, \B)$ and $(Z, \C)$ be three measurable spaces, and consider two Markov kernels $\kappa : X \kernel Y$ and $\eta : Y \kernel Z$. The \emph{composition} of $\kappa$ and $\eta$ is the Markov kernel $\eta \o_k \kappa : X \kernel Z$ defined by:
		$$\forall x \in X, \forall C \in \C, (\eta \o_k \kappa)(x, C) := \int_Y \eta(y, C) \kappa(x, \dy).$$
		If $\mu$ is a probability measure over $(X, \A)$, we likewise define the \emph{composition} of $\mu$ and $\kappa$ as a probability measure over $(Y, \B)$ by:
		$$\forall B \in \B, (\kappa \o_m \mu)(B) := \int_X \kappa(x, B)\mu(\dx).$$
	\end{defi}

	\begin{defi}\label{def:section}
		Let $X$ and $Y$ be two sets and let $A$ be a subset of $X \x Y$. For any $x \in X$, we define the \emph{section of $A$ by $x$} as the set
		$$A_x := \{y \in Y | (x, y) \in A\}.$$
	\end{defi}

	\begin{defi}\label{def:compProd}
		Let $(X, \A)$, $(Y, \B)$ and $(Z, \C)$ be three measurable spaces, and consider two Markov kernels $\kappa : X \kernel Y$ and $\eta : X \x Y \kernel Z$. The \emph{composition-product} of $\kappa$ and $\eta$ is the Markov kernel $\kappa \oxk \eta : X \kernel Y \x Z$ defined by:
		$$\forall x \in X, \forall D \in \B\ox\C, (\kappa \oxk \eta)(x, D) = \int_Y \eta(x, y, D_y) \kappa(x, \dy).$$

		If $\mu$ is a probability measure over $(X, \A)$ and $\kappa$ a Markov kernel from $X$ to $Y$, we define the \emph{composition-product} of $\mu$ and $\kappa$ as a probability measure over $(X \x Y, \A\ox\B)$ by:
		$$\forall C \in \A\ox\B, (\mu \oxm \kappa)(C) = \int_X \kappa(x, C_x) \mu(\dx).$$
		Let us informally describe this last notion. Given a point $x \in X$, we can interpret $\kappa(x)$ as a random point in $Y$. Then $\eta(x, \kappa(x))$ gives a random point in $Z$. The random point $(\kappa \oxk \eta)(x)$ is the random pair $(\kappa(x), \eta(x, \kappa(x)))$.
	\end{defi}

	The name \emph{composition-product} is not standard and is taken from \mathlib, where this operation is called \lean{compProd}. In the literature it is often referred to as \emph{product} or \emph{composition}, but those are used for other operations in Lean and we prefer to stick to \mathlib\ conventions for clarity. The notations $\oxk$ and $\oxm$ are also taken from \mathlib, we use them here to avoid confusion between the two operations.

	To conclude this section, let us recall the definitions of measurable cylinders and the product $\sigma$-algebra, as well as those of a finite product of measures and kernels. We fix $\iota$ an index set and $(X_i, \A_i)_{i \in \iota}$ a family of measurable spaces.

	\begin{nota}
		\phantom{}\\\vspace{-\baselineskip}
		\begin{itemize}
			\item we write $X$ for $\prod_{i \in \iota} X_i$;
			\item if $I \subseteq \iota$, we write $X^I$ for $\prod_{i \in I} X_i$;
			\item if $I \subseteq \iota$, we write $\pi_I$ for the canonical map from $X$ to $X^I$;
			\item if $ J \subseteq I \subseteq \iota$, we write $\pi_{I\to J}$ for the canonical map from $X^I$ to $X^J$.
		\end{itemize}
	\end{nota}

	\begin{defi}\label{def:finite-product-algebra}
		Assume $\iota$ is finite. We define a $\sigma$-algebra over $X$, called the product $\sigma$-algebra and denoted by $\bigotimes_{i \in \iota} \A_i$, as the smallest $\sigma$-algebra containing sets of the form $\prod_{i \in \iota} A_i$,
		where each $A_i$ is in $\A_i$.
	\end{defi}

	\begin{defi}\label{def:measurable-cylinders}
		A \emph{measurable cylinder} is a set of the form $\pi_I\inv(A)$, for a finite set $I \subseteq \iota$ and $A \in \bigotimes_{i \in I} \A_i$.
	\end{defi}

	\begin{defi}\label{def:product-algebra}
		We define a $\sigma$-algebra over $X$, called the product $\sigma$-algebra and denoted by $\bigotimes_{i \in \iota} \A_i$, as the smallest $\sigma$-algebra containing all measurable cylinders. Note that in the case where $\iota$ is finite, this coincides with Definition~\ref{def:finite-product-algebra}.
	\end{defi}

	\begin{lem}\label{lem:ring-cylinders}
		Measurable cylinders form a ring of sets which generates the product $\sigma$-algebra.
	\end{lem}

	\begin{defi}\label{def:prod-measure}
		Assume $\iota$ is finite and let $(X_i, \A_i, \mu_i)_{i \in \iota}$ be a family of probability spaces. There exists a unique measure $\nu$ over $(X, \bigotimes_{i\in\iota} \A_i)$ such that for any family measurable sets $(A_i)_{i \in \iota}$ such that $A_i \in \A_i$,
		\begin{equation*}
			\nu\prth{\prod_{i \in \iota} A_i} = \prod_{i \in \iota} \mu_i(A_i).
		\end{equation*}
		This measure is called the product measure and is denoted by $\bigotimes_{i \in \iota} \mu_i$. If there is only two probability spaces $(X, \A, \mu)$ and $(Y, \B, \nu)$, we denote it by $\mu \otimes \nu$.
	\end{defi}

	\begin{defi}\label{def:prod-kernel}
		Let $(X, \A)$, $(Y, \B)$ and $(Z, \C)$ be three measurable spaces, and consider two Markov kernels $\kappa : X \kernel Y$ and $\eta : X \kernel Z$. The \emph{product} of $\kappa$ and $\eta$ is the Markov kernel $\kappa \times_k \eta : X \kernel Y \times Z$ defined by:
		\begin{equation*}
			\forall x \in X, (\kappa \times_k \eta)(x) = \kappa(x) \otimes \eta(x).
		\end{equation*}
		Once again the notation $\times_k$ is taken from \mathlib.
	\end{defi}

	\subsection{Statement and proof of the theorem}
	We introduce a few more notations.

	\begin{nota}
		If $\iota$ is a set of indices and $(X_i, \A_i)_{i \in \iota}$ a family of measurable spaces, then:
		\begin{itemize}
			\item we write $\A$ for $\bigotimes_{i \in \iota} \A_i$;
			\item if $I \subseteq \iota$, we write $\A^I$ for $\bigotimes_{i \in I} \A_i$.
		\end{itemize}
		If $a$ and $b$ are natural integers we set $\dbrack{a,b} := \{k \in \N | a \le k \le b\}$. If $\iota = \N$, then:
		\begin{itemize}
			\item For $n \in \N$ we write $\Xge{n}$ for $X^{\set{k \in \N \mid k \ge n}}$ and $\Age{n}$ for $\A^{\set{k \in \N \mid k \ge n}}$, with similar notations for $\le, >, <$;
			\item We simply write $\pi_n$ for $\pi_{\dbrack{0,n}}$ and $\pi_{a \to b}$ for $\pi_{\dbrack{0,a} \to \dbrack{0,b}}$.
		\end{itemize}
	\end{nota}

	We now state the theorem.

	\begin{thm}[\cite{Ionescu50}]\label{thm:it}
		Let $(X_n,\A_n)_{n\in\N}$ be a family of measurable spaces. Let $(\kappa_n)_{n\in\N}$ be a family of Markov kernels such that for any $n$, $\kappa_n$ is a kernel from $\Xlen$ to $X_{n+1}$. Then there exists a unique Markov kernel $\xi : X_0 \kernel \Xge{1}$ such that for any $n\ge1$,
		$${\piun}_*\xi = \kappa_0 \oxk ... \oxk \kappa_{n-1}.$$
	\end{thm}

	Let us give first some intuition on the statement of this theorem. Consider a trajectory starting at a point $x_0 \in X_0$. Then $\kappa_0(x_0)$ gives a probability distribution in $X_1$ which represents the distribution of the next position of the trajectory in $X_1$. Say that we pick randomly a point $x_1 \in X_1$ according to this distribution. Then $\kappa_1(x_0, x_1)$ will give the distribution of the next position in $X_2$, and so on. By iterating this reasoning, applying each kernel successively will output the distribution of a trajectory in $\Xint{1,n}$ given the starting point $x_0$. This is formally given by the kernel $\kappa_0 \oxk ... \oxk \kappa_{n-1}$.

	The natural question is then to wonder whether these successive compositions of kernels can be repeated "infinitely many times" in order to get a kernel which given $x_0$ would output the distribution of the infinite trajectory in $\Xge{1}$. The answer is yes, and this kernel is the kernel $\xi$ given by Theorem~\ref{thm:it}.

	By taking each kernel to be constant we easily get the product of a countable family of probability measures, which can then be used to build the product of an arbitrary family of probability measures (this is presented in Section~\ref{sec:prod-measure}). Likewise if we take $X_n = X$ and $\kappa_n = \kappa : X \kernel X$ for all $n$ we get a Markov kernel from $X$ to $\Xge{1}$ which, if we compose it with a measure $\mu$ at the beginning, gives the measure associated to a Markov chain with initial distribution $\mu$ and transition kernel $\kappa$.

	We will now give a sketch of the proof of the theorem to highlight the elements which are required for the formalization. The idea is to apply Carathédory's theorem (Theorem~\ref{thm:carath}) using Lemma~\ref{lem:cont-empty}. This proof can be found at Theorem~8.24 in \cite{Kal2021}. See also Proposition~10.6.1 in \cite{Cohn2013} for a well detailed proof in the particular case of constant kernels.

	In what follows we consider $(X_n, \A_n)_{n\in\N}$ a family of measurable spaces and a family of Markov kernels $(\kappa_n : \Xlen \kernel X_{n+1})_{n\in\N}$, and we set
	$$\eta_{a,b} := \kappa_a \oxk ... \oxk \kappa_{b-1} : \Xle{a} \kernel \Xint{a+1,b}.$$
	Note that, in this simpler setting where $\iota = \N$, measurable cylinders are exactly sets of the form $A \times \Xgtn$, where $n \ge 1$ and $A \in \Alen$.

	Thanks to Lemma~\ref{lem:ring-cylinders}, all we have to do now is to define an additive content over measurable cylinders which we will then extend to the whole $\sigma$-algebra. We also do this for measurable cylinders in $\Xge{n}$ for $n\ge1$ as we will need it for the extension afterwards.

	\begin{defi}\label{def:content}
		Let $a \in \N$ and $x \in \Xlea$. If $A \in \Aint{a+1,b}$, we set
		$$P_a(x, A \x \Xgtb) := \eta_{a,b}(x, A).$$
		This map is well defined because if $b \le c$ then $\eta_{a,c}\prth{x, A \times \Xint{b+1,c}} = \eta_{a,b}(x, A)$.
	\end{defi}

	\begin{lem}\label{lem:add-cont}
		For any $a \in \N$ and any $x \in \Xlea$, the map $P_a(x, \cdot)$ is an additive content over the measurable cylinders of $\Xgt{a}$ (see Definition~\ref{def:add-cont}).
	\end{lem}

	To apply Theorem~\ref{thm:carath} it remains to prove that $P_0(x_0,\cdot)$ is $\sigma$-sub-additive over measurable cylinders. Thanks to Lemma~\ref{lem:cont-empty}, it suffices to show the following lemma:

	\begin{lem}\label{lem:tendsto_zero}
		Let $x_0 \in X_0$. If $(A_n)_{n\in\N}$ is a non-increasing sequence of measurable cylinders with empty intersection then $P_0(x_0, A_n) \tendl{n\to\infty} 0$.
	\end{lem}

	\begin{proof}
		For any $n$ we write $A_n := B_n \times \Xgt{a_n}$, with $B_n \in \Aint{1,a_n}$. In this proof if $A \subseteq \Xge{1}$ and $(x_1, ..., x_n) \in \Xint{1,n}$ we will denote by $A(x_1,...,x_n) := \set{y \in \Xgtn : (x_1,...,x_n,y) \in A}$ the associated section of $A$.

		We use contraposition. Assume there exists $\eps>0$ such that for any $n \in \N$, $P_0(x_0,A_n) \ge \eps$. We wish to prove that the intersection of $A_n$ is nonempty. We define the following sequence of functions:
		\begin{align*}
			f_n : X_1 &\to [0, 1] \\
			x_1 &\mapsto P_1(x_0,x_1, A_n(x_1)).
		\end{align*}
		It allows us to write
		\begin{equation}\label{eq:pzero-eq-intf}
			P_0(x_0,A_n) = \int_{X_1} f_n(x_1) \kappa_0(x_0, \dx_1).
		\end{equation}
		Because $(A_n)$ is non-increasing it is also the case of $(f_n)$. We set $g := \lim_{n\to\infty} f_n$. By the dominated convergence theorem we get $\eps \le P_0(x_0,A_n) \tendl{n\to\infty} \int_{X_1} g(x_1) \kappa_0(x_0,\dx_1)$. As $\kappa_0(x_0, \cdot)$ is a probability measure there exists $x_1 \in X_1$ such that $g(x_1) \ge \eps$, and so for any $n\in\N$ we have $f_n(x_1) \ge \eps$ which implies that $P_1(x_0,x_1, A_n(x_1)) \ge \eps$.

		We can thus repeat the same step replacing $P_0(x_0,\cdot)$ by $P_1(x_0,x_1,\cdot)$ and $A_n$ by $A_n(x_1)$ etc. We get a sequence $x := (x_k)_{k\ge1}$ such that for any $n\in\N$ and $k\ge1$ we have
		$$P_k(x_0,...,x_k,A_n(x_1,...,x_k)) \ge \eps.$$
		We claim that $x \in \bigcap_{n\in\N}A_n$, which will conclude the proof.

		Let $n\in\N$. As $A_n = B_n \times \Xgt{a_n}$ we have that $A_n(x_1,...,x_{a_n})$ is either $\Xgt{a_n}$ or empty according to whether $(x_1,...,x_{a_n})$ belongs to $B_n$ or not. But because $$P_{a_n}(x_0,...,x_{a_n},A_n(x_1,...,x_{a_n})) \ge \eps > 0$$ we have $A_n(x_1,...,x_{a_n}) = \Xgt{a_n}$,
		which proves that $x \in A_n$.
	\end{proof}

	This concludes the main argument of the proof. The measurability and uniqueness can be easily provided by applying the $\pi$-$\lambda$ theorem. We will now focus on those aspects of the proof which proved difficult to formalize and how they were overcome.

	\section{Formalization of the theorem}\label{sec:forma}

	In this section we present different aspects of the proof above which were not straightforward to formalize and the solutions which were found to overcome them. One of the main obstacles was the lack of associativity for the composition-product operation on kernels, as discussed in Section~\ref{sub:comp-ker}. Another obstacle comes from the fact that we are manipulating dependent types. The product type $\prod_{n\in\N} X_n$ is represented in Lean as the type of functions $f$ such that for any $n\in\N$ the image $f(n)$ has type $X_n$. Although the notion of "having a type" can be intuitively linked to "belonging to a set" in set theory, this is actually much more rigid. In particular saying that $f(n)$ has type $X_{n+1-1}$ will not work in Lean because the type $X_{n+1-1}$ is not definitionally equal to $X_n$. We will not go into more details but this brings a significant amount of challenges which were not addressed in the construction with Isabelle/HOL mentioned in the introduction because dependent types do not exist in Isabelle/HOL.

	\subsection{Composition of kernels}\label{sub:comp-ker}

	One of the key tools for the Ionescu-Tulcea theorem is to be able to speak about the composition-product of Markov kernels. Indeed that is necessary even to be able to state the theorem. However formalizing this in a way which allows us to smoothly write the proof above proved to be quite challenging.

	One common rule in formalization is that if we use "..." in a proof it will likely be hard to formalize. In our case the statement of Theorem~\ref{thm:it} itself makes use of these when we write the composition of kernels, and they hide a real problem.

	Transition kernels and their different kinds of composition are defined in \mathlib. However we need to take the composition of a finite family of kernels which does not exist in \mathlib, and here is why. By hand, if we have three Markov kernels $\kappa : X \kernel Y$, $\eta : X \x Y \kernel Z$ and $\xi : X \x Y \x Z \kernel T$ we can define $\kappa \oxk \eta \oxk \xi := (\kappa \oxk \eta) \oxk \xi$. Then we easily get $\kappa \oxk \eta \oxk \xi = \kappa \oxk (\eta \oxk \xi)$. But here we forget a detail: $X \x Y \x Z$ does not mean anything a priori. We can write $(X \x Y) \x Z$ or $X \x (Y \x Z)$ but those are two different types. In paper maths they are canonically isomorphic and we often act as though they were the same but in Lean this isomorphism has to be made explicit.

	Coming back to kernels one can write that $\xi$ is a kernel from $X \x (Y \x Z)$ to $T$. Then $(\kappa \oxk \eta) \oxk \xi$ is a kernel from $X$ to $(Y \x Z) \x T$. But in that case $\kappa \oxk (\eta \oxk \xi)$ is not even well defined because $\xi$ should start from $(X \x Y) \x Z$. If we were to add an isomorphism to define this object anyway the equality $(\kappa \oxk \eta) \oxk \xi = \kappa \oxk (\eta \oxk \xi)$ would still make no sense because on the left we have a kernel from $X$ to $(Y \x Z) \x T$ while on the right it is a kernel from $X$ to $Y \x (Z \x T)$, so they do not have the same type.

	A similar issue arises when formalizing Theorem~\ref{thm:it}: if we consider two Markov kernels $\kappa : X_0 \kernel \Xint{1,a}$ and $\eta : \Xlea \kernel \Xint{a+1,b}$, their composition does not make sense, we need isomorphisms. First to see $\eta$ as a kernel from $X_0 \x \Xint{1,a}$ to $\Xint{a+1,b}$ to take the composition $\kappa \oxk \eta : X_0 \kernel \Xint{1,a} \x \Xint{a+1,b}$, then to see $\kappa \oxk \eta$ as a kernel from $X_0$ to $\Xint{1,b}$. For instance the identity $\eta_{a,b} \oxk \eta_{b,c} = \eta_{a,c}$ that we used repeatedly in the proof does not make sense in Lean. To solve this problem we need to make these isomorphisms explicit and then hide them inside a black box to be able to use it without worrying about the implementation.
	\vspace{\baselineskip}

	Apart from using explicit isomorphisms, a key idea which allowed to develop a black box was the following: instead of considering $\eta_{a,b} := \kappa_a \oxk ... \oxk \kappa_{b-1} : \Xle{a} \kernel \Xint{a+1,b}$, which only makes sense if $a < b$, we consider a more general family of kernels. For any $a, b \in \N$, the kernel $\eta_{a,b}$ will be a Markov kernel from $\Xle{a}$ to $\Xle{b}$. Its input will be some trajectory up to time $a$, and its output will be the distribution of the trajectory up to time $b$ given the trajectory up to time $a$. In particular if $b \le a$ then $\eta_{a, b}$ will just be a deterministic kernel which restricts the trajectory to times smaller than $b$. More generally, if $x \in \Xle{a}$, then $\eta_{a, b}(x)$ will consist of two parts: on $\Xle{a}$ it is just deterministically equal to $x$, and on $\Xint{a+1,b}$ it behaves in the same way as $\kappa_a \oxk ... \oxk \kappa_{b-1}$. The main motivation for this construction is that without any isomorphisms it makes sense to write $\eta_{b,c} \o \eta_{a, b}$, or to wonder whether the latter is equal to $\eta_{a,c}$.

	This idea of only manipulating types of the form $\Xlea$ instead of $\Xint{a,b}$ also significantly reduces the use of different subtypes, avoiding some serious subtype hell. The construction in Isabelle/HOL goes even further than that by only manipulating kernels from $\prod_{n\in\N} X_n$ to $\prod_{n\in\N} X_n$, with different measurability assumptions. If we were to transpose this construction in Lean it would completely get rid of the subtypes issues, but would yield a less practical construction. Indeed, at the end of the day we get a kernel $\xi_a : \Xlea \kernel X$ which given the trajectory up to time $a$ outputs the whole trajectory. If we were to rather consider a kernel from $X$ to $X$ we would need to input a whole trajectory and only keep the first steps. In particular, given a probability measure $\mu$ over $\Xlea$, the measure $\xi_a \o \mu$ will yield a measure over $X$ whose first coordinates follow $\mu$ and the others are obtained by iterating our kernels $\kappa$. This seems to be the most practical way to build a Markov chain given a transition kernel and an initial measure.

	Let us now present the construction. In what follows, we denote by $\delta$ the identity kernel, i.e.\ the Markov kernel which to a point $x$ associates the Dirac mass at $x$ written $\delta_x$. The kernel $\eta_{a, b}$ is defined by induction on $b$. If $b \le a$ then it is deterministic as explained previously. If $a \le b$, we set
	$$\eta_{a,b+1} := (\delta \x_k \kappa_b) \o_k \eta_{a, b}.$$
	This can be interpreted as follows. Consider $x \in \Xle{a}$. Then $\eta_{a, b}(x)$ is the distribution of the trajectory up to time $b$, and we are looking for the trajectory up to time $b + 1$. To do so, we put the identity on the first coordinates and we get the last coordinate by composing with $\kappa_b$.

	Using isomorphisms to be perfectly formal, denote
	$$\phi_b : X_{b+1} \to \prod_{b < k \le b + 1} X_k \qquad \text{and} \qquad \psi_{m, n} : \Xle{m} \x \prod_{m < k \le n} X_k \to \Xlen$$
	the two obvious isomorphisms. Then we define
	\begin{equation}\label{eq:def-eta}
		\eta_{a,b+1} := {\psi_{b,b+1}}_*((\delta \x_k {\phi_b}_*\kappa_b) \o_k \eta_{a, b}).
	\end{equation}

	\newpage

	This construction is formalized as follows in \mathlib :
	\begin{lstlisting}
def partialTraj (κ : (n : ℕ) → Kernel (Π i : Iic n, X i) (X (n + 1)))
		(a b : ℕ) : Kernel (Π i : Iic a, X i) (Π i : Iic b, X i) :=
	if h : b ≤ a then deterministic (frestrictLe₂ h) (measurable_frestrictLe₂ h)
	else @Nat.leRec a (fun b _ ↦ Kernel (Π i : Iic a, X i) (Π i : Iic b, X i))
		Kernel.id
		(fun k _ η_k ↦ ((Kernel.id ×ₖ ((κ k).map (piSingleton k))) ∘ₖ η_k).map
		(IicProdIoc k (k + 1))) b (Nat.le_of_not_le h)\end{lstlisting}

	The declaration \lean{Iic n} stands for $\dbrack{0,n}$.
	One will in particular notice
	\begin{lstlisting}
((Kernel.id ×ₖ ((κ k).map (piSingleton k))) ∘ₖ η_k).map (IicProdIoc k (k + 1))\end{lstlisting}
	which corresponds to Equation~\eqref{eq:def-eta}: \texttt{$\eta$\_k} stands for $\eta_{a,b}$, \texttt{$\kappa$ k} stands for $\kappa_b$, \texttt{Kernel.id} stands for $\delta$, \texttt{piSingleton} stands for $\phi$ and \texttt{IicProdIoc} stands for $\psi$.

	As mentioned above, this construction allows to state and prove intuitive results.
	\begin{lem}\label{lem:eta-comp}
		If $a \le b \le c$ we have $\eta_{b,c} \o_k \eta_{a, b} = \eta_{a,c}$.
	\end{lem}
	\begin{lstlisting}
theorem partialTraj_comp_partialTraj (hab : a ≤ b) (hbc : b ≤ c) :
		partialTraj κ b c ∘ₖ partialTraj κ a b = partialTraj κ a c\end{lstlisting}
	For what follows we need one technical result about $\eta$, and it is the fact that up to an isomorphism it is a product of kernels, the first part being the identity kernel. We don't go into the the details of this lemma, but it is interesting because it gives for free the following result:
	\begin{lem}\label{lem:eta-proj}
		If $b \le c$ then we have ${\pi_{c \to b}}_* \eta_{a,c} = \eta_{a, b}$.
	\end{lem}
	\begin{lstlisting}
theorem partialTraj_map_frestrictLe₂ (a : ℕ) (hbc : b ≤ c) :
	(partialTraj κ a c).map (frestrictLe₂ hbc) = partialTraj κ a b\end{lstlisting}
	In other words, given the distribution of the trajectory up to time $c$, if we restrict it to times smaller than $b$ we get the distribution of the trajectory up to time $b$.

	This family of kernels $\eta$ thus provides a nice tool to state Ionescu-Tulcea theorem. What we will do is that for a given $a \in \N$ we will build a kernel $\xi_a : \Xle{a} \kernel X$ satisfying:
	\begin{equation}\label{eq:xi-proj}
		\forall b \in \N, {\pi_b}_* \xi_a = \eta_{a, b}.
	\end{equation}
	We give more details in Section~\ref{sub:result}.

	\subsection{Projective family of measures}

	The proof of Theorem~\ref{thm:it} heavily relies on Theorem~\ref{thm:carath} which was introduced in \mathlib\ by Rémy Degenne and Peter Pfaffelhuber. They did it to then be able to formalize Kolmogorov's extension theorem, and to do so they provided different tools which were useful to formalize Theorem~\ref{thm:it}. We now present those notions and how they were used in our particular case.

	\begin{defi}\label{def:projective}
		Let $(X_i, \A_i)_{i \in \iota}$ be a family of measurable spaces. Assume that for any finite set $I \subseteq \iota$ we have a measure $\mu_I$ over $X^I$. The measure family $\mu$ is said to be \emph{projective} if for any finite subsets $J \subseteq I \subseteq \iota$ we have ${\pi_{I \to J}}_*\mu_I = \mu_J$.
	\end{defi}

	\begin{defi}\label{def:limite-proj}
		Let $(X_i, \A_i)_{i \in \iota}$ be a family of measurable spaces and let $\mu$ be a projective family of measures over $X$. A \emph{projective limit} of $\mu$ is a measure $\nu$ over $X$ such that for any finite set $I$ we have ${\pi_I}_*\nu = \mu_I$.
	\end{defi}

	These two definitions are respectively called \lean{MeasureTheory.IsProjectiveMeasureFamily} and \lean{MeasureTheory.IsProjectiveLimit} in \mathlib. A key observation is that when given a projective family of measures one can naturally build an additive content over the measurable cylinders of $X$.

	\begin{defi}\label{def:cylinder-content}
		Let $(X_i, \A_i)_{i \in \iota}$ be a family of measurable spaces and let $\mu$ be a projective family of measures over $X$. Let $I$ be a finite subset of $\iota$ and $A \in \A^I$. We set
		$$P(\pi_I\inv (A)) := \mu_I(A).$$
		The projectivity of $\mu$ ensures that this $P$ is well-defined over measurable cylinders of $X$. It is then straightforward to prove that $P$ is an additive content over measurable cylinders.
	\end{defi}

	This content is called \lean{MeasureTheory.projectiveFamilyContent} in \mathlib, and this declaration provides us a straightforward way of formalizing Definition~\ref{def:content}. However the difference in our setup is that we only have a family of measures indexed by $\N$ instead of finite subsets of $\N$, i.e.\ for fixed $a$ and $x \in \Xlea$ we wish to extend the family $(\eta_{a,b}(x))_{b \in \N}$. In order to reuse the projective families formalism available in \mathlib, we introduced a way to extend a family of measures indexed by $\N$ to a family of measures indexed by finite subsets of $\N$. For each $n \in \N$, consider $\mu_n$ a measure over $\Xlen$. Then if $I \subseteq \N$ is finite, we define $\mub_I$ as a measure over $X^I$ by
	\begin{equation*}
		\mub_I := (\pi_{\dbrack{0, \max I} \to I})_* \mu_{\max I}.
	\end{equation*}
	Here is the corresponding declaration in \mathlib :
	\begin{lstlisting}
noncomputable def inducedFamily (μ : ∀ (n : ℕ), Measure (Π i : Iic n, X i))
		(I : Finset ℕ) : Measure ((k : I) → X k) :=
  (μ (I.sup id)).map (restrict₂ I.subset_Iic_sup_id)\end{lstlisting}
	The code \lean{I.sup id} stands for $\max I$ while \lean{(restrict₂ I.subset_Iic_sup_id)} is $\pi_{\dbrack{0, \max I} \to I}$. It is then quite straightforward to show that, if the family $\mu$ is projective in the sense given by Lemma~\ref{lem:eta-proj}, then the family $\mub$ is projective in the sense given by Definition~\ref{def:projective}:
	\pagebreak
	\begin{lstlisting}
theorem isProjectiveMeasureFamily_inducedFamily
    (h : ∀ a b : ℕ, ∀ hab : a ≤ b, (μ b).map (frestrictLe₂ hab) = μ a) :
    IsProjectiveMeasureFamily (inducedFamily μ)\end{lstlisting}
	We can then apply this result to the family $(\eta_{a, b}(x))_{b \in \N}$ (remember that $\eta$ is called \lean{partialTraj} in \mathlib):
	\begin{lstlisting}
lemma isProjectiveMeasureFamily_partialTraj {a : ℕ} (x : Π i : Iic a, X i) :
    IsProjectiveMeasureFamily (inducedFamily (fun b ↦ partialTraj κ a b x))\end{lstlisting}
	Finally, we can use that to define the content of Definition~\ref{def:content}.
	\begin{lstlisting}
noncomputable def trajContent {a : ℕ} (x : Π i : Iic a, X i) :
    AddContent ℝ≥0∞ (measurableCylinders X) :=
  projectiveFamilyContent (isProjectiveMeasureFamily_partialTraj κ x)\end{lstlisting}

	\subsection{Integrating against $\eta$}

	We see that the proof of Lemma~\ref{lem:tendsto_zero} heavily relies on the use of sections of sets. This however is not well suited for formalization purposes because it involves a lot of manipulations with subtypes. For instance, if $A \subseteq \prod_{n\in\N} X_n$ and $x \in X_0$ then $A_x \subseteq \Xge{1}$. So $A_x$ is a set of sequences indexed by $\N_+$ which is not a subset of $\N$ but rather a subtype in Lean. This makes things hard to manipulate because if we want to deal with an element of $\N_+$ we have to provide a proof that it is positive. This would have to be done not only for non-zero integers but also for integers greater than $2$, $3$ and so on with the same troubles each time. Therefore we take a different approach.

	The idea relies on the notion of marginals (\lean{lmarginal} in \mathlib) which is exposed in Section~4 of \cite{Van24}. We adapt their idea as follows: given $a$ and $b$ integers, $f : X \to [0,+\infty]$ and $x \in X$ we define \lean{lmarginalPartialTraj κ a b f x} as
	$$\int_{\Xle{b}} f(y \hookrightarrow_b x) \eta_{a,b}(\pi_a(x), \dy),$$
	where $y \hookrightarrow_b x$ is the same as $x$ but the coordinates indexed by $\dbrack{0,b}$ are those of $y$. Note that in \mathlib\ there are two kinds of integral: the Bochner integral, for functions which take values in a Banach space, and the Lebesgue integral for functions which take values in $[0,+\infty]$. To distinguish them in \mathlib\ we refer to the first as \texttt{integral} and to the second as \texttt{lintegral} (the first \texttt{l} standing for Lebesgue). \texttt{lmarginalPartialTraj} is used for functions taking values in $[0,+\infty]$, hence the first \texttt{l} in the name.

	Informally, this is the same as
	$$\int_{X^{\dbrack{a+1, b}}} f(x_0, ..., x_a, y_{a+1}, ..., y_b, x_{b+1}, ...) (\kappa_a \otimes_k ... \otimes_k \kappa_{b-1})(\pi_a(x), \dy),$$
	because $\eta_{a,b}$ is deterministic for coordinates indexed by $\dbrack{0,a}$. In particular, if $A = B \x \Xgtb$ is a measurable cylinder and we take $f = \ind{A}$, then
	\begin{align*}
		& \texttt{lmarginalPartialTraj $\kappa$ a b f x} \\
		= & \int_{X^{\dbrack{a+1,b}}} \ind{A(x_0, \dots, x_a)}(y_{a+1}, ..., y_b, x_{b+1}, ...) (\kappa_a \otimes_k ... \otimes_k \kappa_{b-1})(\pi_a(x), \dy) \\
		= & P_a(\pi_a(x), A),
	\end{align*}
	where $P_a$ is the content from Definition~\ref{def:content}. What is nice is that we do not need sections of sets to write this anymore. Moreover,
	$$\texttt{lmarginalPartialTraj $\kappa$ a b}$$
	takes as input a map in $X \to [0, +\infty]$ and outputs a map in $X \to [0,+\infty]$ as well. Therefore writing the expression
	$$\texttt{lmarginalPartialTraj $\kappa$ b c (lmarginalPartialTraj $\kappa$ a b f)}$$
	makes sense, in other words we can first integrate between $a$ and $b$ and then between $b$ and $c$. As expected, if $a \le b \le c$ we have
	\begin{lstlisting}
theorem lmarginalPartialTraj_self (hab : a ≤ b) (hbc : b ≤ c)
		{f : (Π n, X n) → ℝ≥0∞} (hf : Measurable f) :
		lmarginalPartialTraj κ a b (lmarginalPartialTraj κ b c f) =
		lmarginalPartialTraj κ a c f\end{lstlisting}
	This in particular allows us to peel off the measure of a measurable cylinder $A := B \x \Xgtn$ as the integral of its indicator: for $x \in X$, we have
	\begin{align*}
		P_0(x_0, A) & = \int_{\Xlen} \ind{A}(y \hookrightarrow_n x) \eta_{0,n}(x_0, \dy) \\
		& = \int_{X^{\le 1}} \int_{\Xlen} \ind{A}(z \hookrightarrow_n x) \eta_{1,n}(y_0, y_1,\dz) \eta_{0,1}(x_0,\dy_0 \dy_1).
	\end{align*}
	This is the formalized version of Equation~\eqref{eq:pzero-eq-intf}. The most important thing to notice is this: thanks to the definition of $\eta$, we do not use sections of sets, and therefore we avoid the use of subtypes. This allows us to formalize Lemma~\ref{lem:tendsto_zero}, which is the main ingredient of the proof, quite smoothly.

	\subsection{The resulting kernel}\label{sub:result}

	As explained at the end of Section~\ref{sub:comp-ker}, instead of producing a kernel $\xi : X_0 \kernel \Xge{1}$ as is written in the statement of Theorem~\ref{thm:it}, we build a family of kernels
	$$\xi_a : \Xle{a} \kernel X,$$
	where $\xi_a$ is obtained as the "limit" of the family of kernels $(\eta_{a, b})_{b \in \N}$. This family of kernels is called \lean{traj} (for "trajectory").
	\begin{lstlisting}
def traj (a : ℕ) : Kernel (Π i : Iic a, X i) (Π n, X n)\end{lstlisting}
	Note that this declaration is a \lean{def} rather than a theorem stating "there exists a kernel such that ...". Indeed we define the kernel thanks to Carathéodory's extension theorem but it remains to prove that it is the projective limit we are looking for, i.e.\ that it satisfies Equation~\eqref{eq:xi-proj}. This follows easily from Definition~\ref{def:cylinder-content} and Lemma~\ref{lem:eta-proj}.
	\begin{prop}\label{prop:xi_proj}
		For any $a, b \in \N$ we have ${\pi_b}_*\xi_a = \eta_{a,b}$.
	\end{prop}
	Actually, because all the measures are finite here this proposition completely characterizes the kernel $\xi_a$, which allows to easily deduce other results. For instance, let $x \in \Xle{a}$. Then $\xi_a(x)$ gives the distribution of the whole trajectory in $X$. On the other hand, $\eta_{a, b}(x)$ gives the distribution of the trajectory up to time $b$. Thus if we compose $\eta_{a, b}(x)$ with the kernel $\xi_b$ we get another distribution for the whole trajectory which we expect to be equal to the one given by $\xi_a(x)$. We have the following lemma:
	\begin{lem}\label{lem:xi-comp-eta}
		If $a \le b$ then $\xi_b \o_k \eta_{a, b} = \xi_a$.
	\end{lem}
	\begin{lstlisting}
theorem traj_comp_partialTraj {a b : ℕ} (hab : a ≤ b) :
		(traj κ b) ∘ₖ (partialTraj κ a b) = traj κ a\end{lstlisting}
	Although these results are quite straightforward, the fact that we were able to formalize them with very short proofs in Lean indicates that the way we chose to represent composition of kernels works well. To illustrate this fact even more, let us consider another result which is not as obvious as the previous ones but was still reasonably easy to formalize.

	\begin{nota}
		For $n \in \N$ we set $\F_n := \pi_n\inv\prth{\Alen}$ the $\sigma$-algebra of events which only depend on coordinates up to $n$.
	\end{nota}

	\begin{thm}
		\label{thm:condexp}
		Let $a, b \in \N$ with $a \le b$. Let $u \in \Xle{a}$. Let $f : X \to \R$ be an integrable function with respect to $\xi_a(u)$. Then for $\xi_a(u)$-almost every $x$, the function $f$ is integrable with respect to $\xi_b(\pi_b(x))$ and
		$$\E_{\xi_a(u)}[f \mid \F_b] (x) = \int_X f(y) \xi_b(\pi_b(x), \dy).$$
	\end{thm}
	In short, the expectation with respect to $\xi_a$ given all the information up to time $b$ is nothing more than the expectation with respect to $\xi_b$. In other words, if a process is distributed according to $\xi_a(u)$, it admits $\xi_b \circ \pi_b$ as a regular conditional distribution given the trajectory up to time $b$.
	\begin{proof}
		We do not present the proof of integrability here and only focus on the computation. By definition of the conditional expectation, it suffices to show that for any $A \in \Ale{b}$,
	$$\int_{\pi_b\inv(A)} \int_X f(y) \xi_b(\pi_b(x), \dy)\xi_a(u,\dx) = \int_{\pi_b\inv(A)} f(x) \xi_a(u, \dx).$$
	We start by writing
	\begin{align*}
		\int_{\pi_b\inv(A)} \int_X f(y) \xi_b(\pi_b(x), \dy)\xi_a(u,\dx) &= \int_A \int_X f(y) \xi_b(x, \dy)({\pi_b}_*\xi_a)(u,\dx) \\
		&= \int_A \int_X f(y) \xi_b(x, \dy)(\eta_{a, b})(u,\dx) \\
		&= \int_{\Xle{b}}\int_X\ind{A}(x)f(y) \xi_b(x, \dy)\eta_{a,b}(u,\dx),
	\end{align*}
	where the second equality is Proposition~\ref{prop:xi_proj}. Now we want to use the fact that $\xi_b(x)$-a.s., $\pi_b(y)=x$ to get
	\begin{align*}
		\int_{\Xle{b}}\int_X\ind{A}(x)f(y) \xi_b(x, \dy)\eta_{a,b}(u,\dx) &= \int_{\Xle{b}}\int_X\ind{A}(\pi_b(y))f(y) \xi_b(x, \dy)\eta_{a,b}(u,\dx).
	\end{align*}

	However this is not as straight-forward as it seems. Indeed, it is not possible to prove in Lean that $\xi_b(x)$-a.s., $\pi_b(y)=x$. The reason is that this would amount to prove that $\delta_x(\set{y \,|\, y \ne x}) = 0$, and this can only be proved if this set is measurable. However in our setup we are working on an arbitrary $\sigma$-algebra which cannot ensure this measurability.

	The argument above is still correct though, in particular because we are integrating a measurable function and we thus have measurability hypotheses. We solve the issue by proving the following result:
	\begin{lem}
		$\eta_{a, b}(u) \oxm \xi_b = (y \mapsto (\pi_b(y), y))_*\xi_a(u)$.
	\end{lem}
	\begin{lstlisting}
lemma partialTraj_compProd_traj {a b : ℕ} (hab : a ≤ b) (u : Π i : Iic a, X i) :
		(partialTraj κ a b u) ⊗ₘ (traj κ b) =
		(traj κ a u).map (fun x ↦ (frestrictLe b x, x)) \end{lstlisting}
	This result is not too surprising. On the left-hand side, we start with a random trajectory up to time $b$ given by $\eta_{a, b}(u)$, let us call it $x$, and we plug $x$ into $\xi_b$ to get another random trajectory $y$. In the end we output the couple $(x, y)$. On the right-hand side, we take a random trajectory $y$ given by $\xi_a(u)$ and output $(\pi_b(y), y)$. The identity $\xi_b \o_k \eta_{a, b} = \xi_a$ allows to conclude.

	We can finish the computation as follows:
	\begin{align*}
		\int_{\Xle{b}}\int_X\ind{A}(x)f(y) \xi_b(x, \dy)\eta_{a,b}(u,\dx) &= \int_{\Xle{b} \x X} \ind{A}(x)f(y) (\eta_{a, b}(u) \oxm \xi_b)(\dx\dy) \\
		&= \int_{\Xle{b} \x X} \ind{A}(x)f(y) ((y \mapsto (\pi_b(y), y))_*\xi_a(u))(\dx\dy) \\
		&= \int_X \ind{A}(\pi_b(y))f(y)\xi_a(u,\dy) \\
		&= \int_{\pi_b\inv(A)} f(y) \xi_a(u,\dy).\qedhere
	\end{align*}
	\end{proof}

	\section{Infinite product of probability measures}\label{sec:prod-measure}

	To conclude this paper we give a quick description of how Theorem~\ref{thm:it} can be used to build the product of an arbitrary family of probability measures. The product of finitely many measures has been available in \mathlib\ for a long time now. It is based on iterating the construction for the product of two measures, whose formalization is exposed in Section~4 of \cite{VanD2021}. As mentioned in Section~\ref{sec:proba}, by taking constant kernels in Theorem~\ref{thm:it} we easily get the product of a countable family of probability measures, and this can be used to build the product of an arbitrary family.
	\begin{thm}\label{thm:prod-measure}
		Let $(X_i, \A_i, \mu_i)_{i\in\iota}$ be a family of probability spaces. Then there exists a unique measure on $(X, \A)$, called the \emph{product measure} and denoted by $\bigotimes_{i\in\iota} \mu_i$, such that for any finite set $I \subseteq \iota$, we have
		\begin{equation}\label{eq:def-prod}
			{\pi_I}_* \bigotimes_{i\in\iota} \mu_i = \bigotimes_{i\in I} \mu_i.
		\end{equation}
	\end{thm}

	\begin{proof}
		Given $A := \pi_J\inv(A')$ a measurable cylinder, we set $P(A) := \bigotimes_{j\in J} \mu_j(A')$. We want to prove that $P$ can be extended to a measure over $\A$, which will conclude the proof because this extension will satisfy Equation~\eqref{eq:def-prod} by definition of $P$. To prove that this extension exists we apply Lemma~\ref{lem:cont-empty}. Let $(A_n)_{n\in\N}$  be an non-increasing family of measurable cylinders such that $\bigcap_{n\in\N}A_n =\empty$. For any $n$, $A_n$ only depends on a finite set of coordinates $I_n \subseteq \iota$. We set $I := \bigcup_{n\in\N} I_n$, which is a countable set. Because $A_n$ is a measurable cylinder there exists $B_n \in \A^{I_n}$ such that $A_n = \pi_{I_n}\inv(B_n)$. On the other hand, for any $i \in \iota$, $X_i$ is a probability space and thus is nonempty. Let's denote $(z_i)_{i \in \iota}$ an element of $X$. Then we can define $\phi : X^I \to X$ which to some $(x_i)_{i \in I}$ associates $(y_i)_{i \in \iota}$, with the rule that if $i \in I$ then $y_i = x_i$, otherwise $y_i = z_i$. We clearly have that for any $n$, $\pi_{I \to I_n}\inv(B_n) = \phi\inv(A_n)$. Therefore we get:
		$$P(A_n) = \bigotimes_{i\in I_n} \mu_i(B_n) = {\pi_{I \to I_n}}_*\bigotimes_{i\in I} \mu_i(B_n) = \bigotimes_{i\in I} \mu_i(\pi_{I \to I_n}\inv(B_n)) = \bigotimes_{i\in I} \mu_i(\phi\inv(A_n)).$$
		Note that the second equality follows from the existence of the product measure $\bigotimes_{i\in I} \mu_i$, which itself follows from the Ionescu-Tulcea theorem in the case of constant kernels. We conclude by monotonous continuity of $\bigotimes_{i\in I} \mu_i$ to get that $P(A_n) \tendl{n\to\infty} 0$.

		Uniqueness follows from the $\pi$-$\lambda$ theorem.
	\end{proof}

	The formalization of Theorem~\ref{thm:prod-measure} went quite smoothly. The most annoying part actually was to prove that the Markov kernel $\eta_{a, b}$ defined by Equation~\eqref{eq:def-eta} does reduce to the product measure in the case of constant kernels. Here is the precise statement.
	\begin{lstlisting}
theorem partialTraj_const {a b : ℕ} :
		partialTraj (fun n ↦ const _ (μ (n + 1))) a b =
		(Kernel.id ×ₖ (const _ (Measure.pi (fun i : Ioc a b ↦ μ i)))).map
			(IicProdIoc a b)\end{lstlisting}
	It mostly relies on proving the obvious statement that if $a \le b \le c$, then up to an isomorphism
	$$\bigotimes_{a < i \le c} \mu_i = \bigotimes_{a < i \le b} \mu_i \ox \bigotimes_{b < i \le c} \mu_i.$$

	\section{Conclusion}

	We presented the formalization of the Ionescu-Tulcea theorem. The interest of this project is twofold. First, it is an important result of probability theory which allows to build arbitrary families of independent random variables, and such objects are crucial if we hope to be able to formally build more complicated objects such as the Brownian motion. Second, it is interesting to see that some of the biggest obstacles in the formalization arose from some of the most basic aspects of the theorem, particularly the definition of the composition-product of multiple Markov kernels. This illustrates how some rather simple concepts have to be carefully translated when doing formalization, so as to not lose any mathematical meaning but also provide the smoothest API possible. In this context, the API is the set of simple lemmas which allows to manipulate an object without knowing the details of its implementation. It is crucial to provide a nice interface for other users to be able to use the object in question. Although we achieved what seems to be a tractable API, we need to wait for other \mathlib\ users to try and use it to get more certainty that the formalization of the theorem went the right way.

	\section*{Acknowledgements}

	The author would like to thank: Sébastien Gouëzel, who proposed this project as part of an internship and gave precious advice, Rémy Degenne for reviewing the code and allowing it to make its way to \mathlib, and all the \mathlib\ contributors who allowed this project to reach its goal.

	\bibliographystyle{apalike}
	\bibliography{biblio}
\end{document}